\documentclass{amsart}
\usepackage{amsmath,amsthm,amssymb,verbatim}

\newtheorem{thm}{Theorem}
\newtheorem{prop}[thm]{Proposition}
\newtheorem{lem}[thm]{Lemma}

\newtheorem{conj}[thm]{Conjecture}

\theoremstyle{definition}

\theoremstyle{remark}

\newcommand{\del}{\partial} 
\newcommand{\dbar}{\overline{\del}}
\newcommand{\ddb}{i\del\dbar}

\title{Greatest lower bounds on the Ricci curvature of Fano manifolds}
\author{G\'abor Sz\'ekelyhidi}
\date{}

\begin{document}
\begin{abstract}
	On a Fano manifold $M$ 
	we study the supremum of the possible $t$ such that
	there is a K\"ahler metric $\omega\in c_1(M)$ with Ricci
	curvature bounded below by $t$. This is shown to be
	the same as the maximum existence time of Aubin's
	continuity path for finding K\"ahler-Einstein metrics. We show
	that on $\mathbf{P}^2$ blown up in one point this supremum is
	$6/7$, and we give upper bounds for other manifolds. 
\end{abstract}

\maketitle

\section{Introduction}
The problem of finding K\"ahler-Einstein metrics is a fundamental one in
K\"ahler geometry. After the works of Yau~\cite{Yau78} and
Aubin~\cite{Aub78}
what remained is settling the existence question for Fano manifolds.
Yau~\cite{Yau93} conjectured that in this case the existence is related to
stability of the manifold in the sense of geometric invariant theory.
Important progress was made by Tian~\cite{Tian97}, who introduced the
notion of K-stability. This was extended by Donaldson to
the study of more general constant scalar curvature K\"ahler metrics
(see eg. \cite{Don01}, \cite{Don08_1}). 
The conjecture relating K-stability to the existence of constant scalar
curvature K\"ahler metrics, now called the Yau-Tian-Donaldson
conjecture, is
currently a very active field of research. For a survey and many more
references, see~\cite{PS08}. 

In this paper we study Aubin's~\cite{Aub84}
continuity method for finding
K\"ahler-Einstein metrics. Given a K\"ahler metric $\omega\in c_1(M)$,
this approach is to find $\omega_t$ solving
\[ \mathrm{Ric}(\omega_t) = t\omega_t + (1-t)\omega \]
for all $t\in[0,1]$. For $t=0$ a solution exists by Yau's
theorem. This continuity path has nice properties, an important 
one being that the Mabuchi energy~\cite{Mab86} is monotonically
decreasing along the path. This was exploited in~\cite{BM85} to show
the lower boundedness of the Mabuchi energy. It is also crucial for
finding a priori estimates using properness of the Mabuchi functional
(see~\cite{Tian97}).

We are interested in the situation when we cannot solve up to $t=1$.
Clearly understanding this is crucial in the study of obstructions to
the existence of K\"ahler-Einstein metrics. A natural question is what
the supremum of the $t$ is for which we can solve the equation. We first
show that this is independent of the choice of $\omega$, and is equal to
the invariant $R(M)$ that we define by
\[ R(M) = \sup_{t\in[0,1]}
\{\exists\omega\in c_1(M)\text{
such that Ric}(\omega)>t\omega\}.\]
The proof 
goes via relating the existence of a solution to properness of a
certain functional. For the case $t=1$ this has been done in
\cite{Tian97},
with a stronger version of properness shown in \cite{PSSW}.

The problem then becomes to determine $R(M)$ for
manifolds which do not admit K\"ahler-Einstein metrics.
Tian~\cite{Tian92}
considered the problem of bounding $R(M)$ and obtained the upper bound
$R(M_1)\leqslant 15/16$ where $M_1$ is $\mathbf{P}^2$ blown up in one point. 
We show that in fact $R(M_1)=6/7$. More generally we give an upper bound
for any Fano manifold which has non-trivial vector fields and non-vanishing 
Futaki invariant. In Section~\ref{sec:test-config} we show that if
$R(M)=1$ then the manifold is K-semistable with respect to
test-configurations with smooth total space. 

\subsection*{Acknowledgements}
I would like to thank Jacopo Stoppa and Valentino Tosatti for helpful
discussions.

\section{The definition of $R(M)$}
Let $M$ be a Fano manifold, so $c_1(M)>0$.
Let us fix a base metric $\eta\in c_1(M)$, and consider a family of
metrics 
\[ \omega_t = \eta + \ddb\phi_t.\]
The Mabuchi functional~\cite{Mab86} is defined by its variation
\[ \frac{d}{dt}\mathcal{M}(\omega_t) = \int_M \dot{\phi_t}(n -
S(\omega_t))\omega_t^n,\]
normalised so that $\mathcal{M}(\eta)=0$.
Here $S(\omega_t)$ is the scalar curvature. 
For any K\"ahler metric $\alpha\in c_1(M)$ we also define the
functional $\mathcal{J}_\alpha$, by its variation
\[ \frac{d}{dt}\mathcal{J}_\alpha(\omega_t) = \int_M \dot{\phi_t}(
\Lambda_{\omega_t}\alpha - n)\omega_t^n,\]
normalised so that $\mathcal{J}_\alpha(\eta)=0$. Here
$\Lambda_{\omega_t}$ means the trace with respect to $\omega_t$. 
The functional $\mathcal{J}_\alpha$ is essentially the same as 
$I-J$ in terms of Aubin's $I,J$ functionals (see \cite{BM85}).
When $\alpha$ is not necessarily in the same K\"ahler class as $\omega$,
it was introduced in \cite{Chen00} to study the Mabuchi energy on
manifolds with $c_1 < 0$. See also \cite{Wei04}, \cite{SW04}.

Given any $\omega\in c_1(M)$, 
Aubin's continuity path for finding K\"ahler-Einstein
metrics is given by 
\begin{equation} \label{eq:cont} 
	\omega_t^n = e^{h_\omega - t\phi_t}\omega^n, 
\end{equation}
where $h_\omega$ 
is the Ricci potential, defined by
\[ Ric(\omega) - \omega = \ddb h_\omega,\]
and normalised so that $\int_M e^{h_\omega}\omega^n = \int_M \omega^n$.
Equivalently we have $Ric(\omega_t)=t\omega_t + (1-t)\omega$. 
For $t=0$ this can be solved by Yau's theorem~\cite{Yau78}.

Finally we call a functional $\mathcal{F}$ defined on the space of
K\"ahler metrics in $c_1(M)$ \emph{proper} if there exist constants
$\epsilon,C>0$ such that
\[ \mathcal{F}(\omega) > \epsilon\mathcal{J}_\eta(\omega) - C\]
for all $\omega\in c_1(M)$. Since $\mathcal{J}_\eta$
is the same as the functional $I-J$ in the literature, this notion of
properness conincides with the one used in \cite{Tian97}.

\begin{thm} The following are equivalent for $0\leqslant t<1$.
	\begin{itemize}
		\item We can solve Equation (\ref{eq:cont}).
		\item There exists a metric $\omega\in c_1(M)$ such that
			$\mathrm{Ric}(\omega) > t\omega$.
		\item The functional $\mathcal{M} +
			(1-t)\mathcal{J}_\omega$ is proper
			for any $\omega\in c_1(M)$. 
	\end{itemize}
	In particular we can introduce an invariant $R(M)$ to be the
	supremum of the possible $t < 1$ for which the above statments
	hold.
\end{thm}

The proof of the theorem follows from Lemmas \ref{lem:diff},
\ref{lem:proper} and \ref{lem:cont}.
The statement of the theorem for $t=1$ (the second statement replaced
with $\mathrm{Ric}(\omega)=\omega$) follows from the works of
Tian~\cite{Tian97} and Phong-Song-Sturm-Weinkove~\cite{PSSW}. Note that by the
following lemma all the $\mathcal{J}_\omega$ for different $\omega$ are
equivalent, so by definition they are all proper. The invariant $R(M)$
measures what the smallest multiple of $\mathcal{J}_\omega$ is that we
need to add to $\mathcal{M}$ to make it proper. 

\begin{lem}\label{lem:diff}
	If $\alpha,\alpha'$ are in the same K\"ahler class then for
	all $\omega\in c_1(M)$, we have
	\[ |(\mathcal{J}_{\alpha}-\mathcal{J}_{\alpha'})
	(\omega)| < C \]
	for some constant $C$ independent of $\omega$. 
\end{lem}
\begin{proof}
	Let us write $\alpha=\alpha'+\ddb\psi$, and $\omega = \eta +
	\ddb\phi$. Writing $\omega_t = \eta + t\ddb\phi$ we have
	\[ \begin{aligned}
		\frac{d}{dt}(\mathcal{J}_\alpha-\mathcal{J}_{\alpha'})(\omega_t)
		&= \int_M
		\phi\Lambda_{\omega_t}(\alpha-\alpha')\omega_t^n \\
		&= n \int_M \phi(\ddb\psi)\wedge (\eta+t\ddb\phi)^{n-1}.
	\end{aligned}\]
	We can then compute
	\[ \begin{aligned}
		(\mathcal{J}_\alpha-\mathcal{J}_{\alpha'})(\omega)
		&= n \int_0^1\int_M
		\phi(\ddb\psi)\wedge(t\omega + (1-t)\eta)^{n-1}\,dt \\
		&= n \int_0^1\int_M \psi(\ddb\phi)
		\wedge\sum_{p=1}^n\binom{n-1}{p-1} t^{p-1}(1-t)^{n-p}
		\omega^{p-1}\wedge\eta^{n-p}\,dt\\
		&= \int_M\psi(\omega-\eta) \wedge\sum_{
		p=1}^n \omega^{p-1}\wedge\eta^{n-p}\\
		&= \int_M\psi(\omega^n-\eta^n),
	\end{aligned} \]
	which is uniformly bounded in terms of $\sup|\psi|$.
\end{proof}

The following proposition, which follows directly from the work of
Chen-Tian~\cite{CT05_1} is the key technical result.

\begin{prop}\label{lem:bdd}
	If $\omega$ satisfies the equation
	\begin{equation}\label{eq:twist}
		\mathrm{Ric}(\omega) = t\omega + (1-t)\alpha,
	\end{equation}
	where $\alpha\in c_1(M)$ is positive, then the functional
	$\mathcal{M} + (1-t)\mathcal{J}_\alpha$ is bounded below.
\end{prop}
\begin{proof}
	First note that $\omega$ satisfying Equation~(\ref{eq:twist}) is
	a critical point of the functional
	$\mathcal{M}+(1-t)\mathcal{J}_\alpha$. This follows directly
	from the variational formula
	\[ \frac{d}{dt}[\mathcal{M}(\omega_t) +
	(1-t)\mathcal{J}_\alpha(\omega_t)] = \int_M \dot{\phi_t}[ tn +
	(1-t)\Lambda_{\omega_t}\alpha - S(\omega_t)]\,\omega^n\]
	and taking the trace of Equation~(\ref{eq:twist}).

	By Yau's theorem~\cite{Yau78} we can find a metric $\omega_0\in
	c_1(M)$ such that $\mathrm{Ric}(\omega_0)=\alpha$. By the
	same computation as in Chen~\cite{Chen00}, we have
	\[ \mathcal{M}(\omega) + (1-t)\mathcal{J}_\alpha(\omega) =
	D + \int_M \log\frac{\omega^n}{\omega_0^n}\,\omega^n -
	t\mathcal{J}_\alpha(\omega),\]
	for some constant $D$. As in Chen-Tian, Theorem 6.1.1. this
	functional is weakly sub-harmonic on almost smooth solutions of
	the geodesic equation in the space of K\"ahler metrics. Then the
	argument in Theorem 6.2.1. implies that the functional is
	bounded below on the space of metrics in the first Chern class.
\end{proof}

\begin{lem}\label{lem:proper}
	If there exists a metric $\omega$ with
	$\mathrm{Ric}(\omega)> t\omega$, then the functional 
	$\mathcal{M} + (1-t)\mathcal{J}_\eta$ is proper.
\end{lem}
\begin{proof}
	Let us write
	\[ \mathrm{Ric}(\omega) = t\omega + (1-t)\alpha,\]
	where $\alpha$ is a positive form in $c_1(M)$. It
	follows from the previous proposition 
	that the functional $\mathcal{M} + (1-t)\mathcal{J}_\alpha$
	is bounded from below. By
	Lemma~\ref{lem:diff} it follows that
	$\mathcal{M}+(1-t)\mathcal{J}_\eta$ is also bounded from below.

	In order to show that it is proper, we use a perturbation
	argument. We want to show that for sufficiently small
	$\epsilon>0$ we can find $\omega'$ such that
	\[ \mathrm{Ric}(\omega') - (t+\epsilon)\omega' =
	(1-t-\epsilon)\alpha.\]
	This is just the openness statement in Aubin's continuity
	method~\cite{Aub84}.
	Then the previous argument implies that $\mathcal{M} +
	(1-t-\epsilon)\mathcal{J}_\eta$ is bounded below, so
	\[ \mathcal{M}+(1-t)\mathcal{J}_\eta = \epsilon\mathcal{J}_\eta
	+ (\mathcal{M} + (1-t-\epsilon)\mathcal{J}_\eta) >
	\epsilon\mathcal{J}_\eta - C,\]
	which is what we wanted to prove.
\end{proof}

\begin{lem} \label{lem:cont}
	If the functional $\mathcal{M}+  (1-s)\mathcal{J}_\eta$ is
	proper, then for any metric $\omega\in c_1(M)$ we can find an
	$\omega_s$ such that
	\[ \mathrm{Ric}(\omega_s) = s\omega_s + (1-s)\omega,\]
	that is, we can solve along the continuity method up to time $s$. 
\end{lem}
\begin{proof} This is a slight extension of a result in~\cite{Tian97}
	(see also~\cite{BM85}).
	Using Yau's estimates~\cite{Yau78} we only
	need to show that if the path
	of metrics
	$\omega_t = \omega+\ddb\phi_t$ satisfies
	\begin{equation}\label{eq:MA}
		\omega_t^n = e^{h_\omega-t\phi_t}\omega^n
	\end{equation}
	for $t < s$, then there is a uniform $C^0$ bound $\sup|\phi_t| <
	C$. For this we compute the derivative
	\[ \frac{d}{dt}\big[\mathcal{M}(\omega_t) +
	(1-s)\mathcal{J}_\omega(\omega_t)\big].\]
	Differentiating Equation (\ref{eq:MA}) we get
	\[ \Delta_t\dot{\phi_t} = -\phi_t-t\dot{\phi_t},\]
	where $\Delta_t$ is the Laplace operator of the metric
	$\omega_t$. Using the formula
	\[ S(\omega_t) = tn + (1-t)\Lambda_{\omega_t}\omega,\]
	we can compute 
	\[ \begin{aligned}
		\frac{d}{dt}\big[\mathcal{M}(\omega_t) +
		(1-s)\mathcal{J}_\omega (\omega_t)\big] &= \int_M
		\dot{\phi_t} \big[ -(1-t)\Lambda_{\omega_t}\omega +
		(1-s)\Lambda_{\omega_t}\omega - c\big]\omega_t^n \\
		&= (s-t)\int_M\dot{\phi_t}\Lambda_{\omega_t}(
		\omega_t-\omega)\,\omega_t^n \\
		&= (s-t)\int_M\dot{\phi_t}\Delta_t\phi_t\,\omega_t^n\\
		&= (s-t)\int_M (-\phi_t-t\dot{\phi_t})\phi_t\,\omega_t^n
		\\ &= (s-t)\left[-\int_M\phi_t^2\omega_t^n + t\int_M
		\dot{\phi_t}(\Delta_t\dot{\phi_t} + t\dot{\phi_t})
		\omega_t^n\right] \\
		&\leqslant 0,
	\end{aligned}\]
	as long as $t < s$.
	Here we have used that $\Delta_t + t$ is a negative operator
	since $\mathrm{Ric}(\omega_t)\geqslant t\omega_t$.
	
	Since $\mathcal{M} + (1-s)\mathcal{J}_\omega$ is proper (again
	using Lemma~\ref{lem:diff} to relate the different $\mathcal{J}$
	functionals), we obtain a uniform bound
	\[ \mathcal{J}_\omega(\omega_t) < C\]
	for $t < s$. As in \cite{BM85} this gives the required $C^0$
	estimate.
\end{proof}

\section{Bounding the invariant $R(M)$}
It is an interesting problem to find bounds on $R(M)$ for a given
Fano manifold $M$. First let us briefly discuss lower bounds.
Clearly when $M$ admits a K\"ahler-Einstein metric
then $R(M)=1$. The converse however is not true. For instance the
unstable deformations of the Mukai 3-fold given by Tian~\cite{Tian97} have
$R(M)=1$. To see this first recall that
the Mukai 3-fold $M_0$ admits a K\"ahler-Einstein
metric (see Donaldson~\cite{Don07}), so for any $t < 1$ there is a metric
$\omega_0$ on $M_0$ with 
\[\mathrm{Ric}(\omega_0) > t\omega_0.\]
Tian's example is a 
manifold $M$ such that $M_0$ has arbitrarily small deformations which
are biholomorphic to $M$ (there exists a degeneration of $M$ to
$M_0$). With such small deformations we can obtain a metric $\omega$ on
$M$, such that $\mathrm{Ric}(\omega)>t\omega$ still holds. Since we can
do this for any $t$, this implies
that $R(M)=1$. Alternatively, it is well-known that $R(M)=1$ if the Mabuchi 
energy is bounded from below (see \cite{BM85}), and Chen~\cite{Chen08} 
showed
that this is the case for the manifold $M$. 
More generally we have the following.
\begin{prop}\label{prop:def}
	If $M$ is a K\"ahler-Einstein manifold and $M'$ is a
	sufficiently small deformation of the complex structure of $M$,
	then $R(M')=1$. 
\end{prop}
\begin{proof}
	This follows from the proof of the main result in \cite{GSz08}.
	It is shown there that there exists a small ball
	$B\subset\mathbf{C}^k$ with a linear
	action of the group of holomorphic automorphisms $Aut(M)$ on $B$
	such that points in a complex analytic subset $Z\subset B$ give
	all the 
	small deformations of the complex structure of $M$ which have
	the same first Chern class as $M$ (ie. $Z$ is a subset of the 
	Kuranishi
	space~\cite{Kur65} of $M$) and manifolds in the same $Aut(M)$
	orbit are biholomorphic. Moreover the points in $Z$ which are
	polystable for the action of $Aut(M)$ (ie. their orbit is closed
	in $\mathbf{C}^k$) correspond to deformations of $M$ which admit
	K\"ahler-Einstein metrics. Suppose that the small deformation
	$M'$ corresponds to a point $z\in Z$. Either $z$ is polystable,
	in which case $M'$ admits a K\"ahler-Einstein metric, or there
	exists a polystable point $z_0$ 
	in the closure of the $Aut(M)$-orbit of
	$z$, such that also $z_0\in Z$. This $z_0$ is obtained by
	minimizing the norm over the $Aut(M)$-orbit of $z$. Let $M_0$ be
	the manifold corresponding to $z_0$ (it may be that $M_0=M$), so
	$M_0$ admits a K\"ahler-Einstein metric.
	Since $z_0$ is in the closure of the orbit of $z$, we can
	realise $M'$ as an arbitrarily small deformation of $M_0$. The
	above argument then shows that $R(M')=1$. 
\end{proof}

In addition one can give a lower bound in terms of the alpha
invariant $\alpha(M)$ or its equivariant version (Tian~\cite{Tian87}),
namely $R(M)\geqslant \alpha(M)\cdot \frac{n+1}{n}$ as long as this is
no greater than 1, where $n$ is the
complex dimension. 

There is much less known about upper bounds for $R(M)$.
The problem was briefly studied in the paper of
Tian~\cite{Tian92}, and he found some bounds in terms of the tangent bundle.
For $\mathbf{P}^2$ blown up in one point he found the upper bound
$15/16$. In the next section we will show that in fact 
$R(M_1)=6/7$ where $M_1$
is $\mathbf{P}^2$ blown up in one point. For the blowup in 2 points we
show $R(M_2)\leqslant 21/25$. 

To obtain upper bounds we can
use the recent work of Stoppa~\cite{Sto08_1}. The
basic observation is that the equation
\[ \mathrm{Ric}(\omega) = t\omega + (1-t)\alpha \]
is a twisted cscK equation (or generalised K\"ahler-Einstein equation in
the terminology of Song-Tian~\cite{ST07}). Stoppa gives an obstruction to
solving this equation, generalising the slope stability obstruction to
the existence of cscK metrics due to Ross-Thomas~\cite{RT06}. As we will
see this gives a
good bound for $\mathbf{P}^2$ blown up in 1 point, but for the blowup in
2 points it does not give anything because it is slope stable (see
Panov-Ross~\cite{PR07}). So we now give another upper bound
which in some sense is more basic. Both are based on constructing
sequences of metrics along which $\mathcal{M}+(1-t)\mathcal{J}_\omega$
is not bounded from below for certain $t$. Stoppa uses a metric
degeneration which models deformation to the normal cone, whereas we
look at one parameter families of metrics induced by holomorphic
vector fields. 

\begin{prop}\label{prop:limJ}
	Fix a metric $\omega$ such
	that $\mathrm{Ric}(\omega)=\alpha$ is a positive form. Let $H$
	be a smooth real valued function on $M$ and 
	suppose that $X=\nabla H$ is a holomorphic
	vector field. Write $f_t:M\to M$ for the one-parameter group of
	diffeomorphisms generated by $X$. Let $\omega_t=f_t^*\omega$. Then
	\[ \lim_{t\to\infty} \frac{d}{dt} \mathcal{J}_\alpha(\omega_t) =
	\int_M H(S(\omega)-n)\,\omega^n + \lim_{t\to\infty} \int_M
	(f_t^{-1})^*(\Delta H)\,\omega^n.\]
	Here $\Delta$ is the Laplacian with respect to the metric
	$\omega$. 

	It follows that 
	\[ \lim_{t\to\infty} \frac{d}{dt}
	\left(\mathcal{M}(\omega_t) + (1-s)
	\mathcal{J}_\alpha(\omega_t)\right) = s\int_M
	H(n-S(\omega))\,\omega^n + (1-s)K\,\mathrm{Vol}(M),\]
	where $K$ is the divergence of the vector field $X$ on 
	the submanifold where $H$ achieves its minimum. If this limit is
	negative, then $\mathcal{M}+(1-s)\mathcal{J}_\alpha$ is not
	bounded below, and so $R(M)\leqslant s$.
\end{prop}
\begin{proof}[Proof of Proposition \ref{prop:limJ}]
	Let $\omega_t = f_t^*\omega$ and
	write $\omega_t=\omega+\ddb\phi_t$. Then $\dot{\phi_t}= f_t^*H$.
	We compute
	\[\begin{aligned}
		\frac{d}{dt} \mathcal{J}_\alpha(\omega_t) &=
		\int_M\dot{\phi_t}\left(\Lambda_{\omega_t}\alpha -
		n\right)\,\omega_t^n \\ &= \int_M\dot{\phi_t}
		\left[\Lambda_{\omega_t}(\mathrm{Ric}(\omega)-\mathrm{Ric}
		(\omega_t))\right]\,\omega_t^n + \int_M \dot{\phi_t}
		(S(\omega_t) - n)\,\omega_t^n.
	\end{aligned}\]
	The second term is simply 
	\[ \int_M H(S(\omega) - n)\,\omega^n.\]
	For the first term we have
	\[ \begin{aligned}
		\int_M \dot{\phi_t}\left[\Lambda_{\omega_t}(\mathrm{Ric}
		(\omega)-\mathrm{Ric}(\omega_t))\right]\,\omega_t^n &=
		\int_M \dot{\phi_t}\Delta_t\log\frac{\omega_t^n}{
		\omega^n}\,\omega_t^n \\
		&= \int_M (\Delta_t\dot{\phi_t}) \log\frac{\omega_t^n}{
		\omega^n}\,\omega_t^n \\
		&= \frac{d}{dt}\int_M \log\frac{\omega_t^n}{\omega^n}
		\,\omega_t^n = 
		\frac{d}{dt} \int_M \log\frac{\omega^n}{ \omega_{-t}^n}
		\omega^n.
	\end{aligned}\]
	We have written $\omega_{-t} = (f_t^{-1})^*\omega$. Then
	$\frac{d}{dt}\omega_{-t} = -\ddb (f_t^{-1})^*H$ since it is
	the same as flowing along $-\nabla H$. Therefore we obtain
	\[ \frac{d}{dt}\int_M -\log\frac{\omega_{-t}^n}{\omega^n}
	\,\omega^n = \int_M (f_t^{-1})^*(\Delta H)\,\omega^n.\]
	The first part of the result follows.

	For convenience let us assume that $\inf H=0$. 
	Note first of all that $J\nabla H$ is a
	Killing field, and so it generates a torus action. In particular
	$H$ is a component of the moment map for a torus action. It
	follows that $H$ is a Morse-Bott function with even-dimensional
	critical manifolds of even index (see McDuff-Salamon~\cite{MS98}),
	and so $H^{-1}(0)$ is a connected complex submanifold, and in
	addition 
	\[ \lim_{t\to\infty} f_t^{-1}(x) \in H^{-1}(0) \]
	for a dense open set in $M$. 
	For a point $y\in H^{-1}(0)$, the Laplacian $\Delta H(y)$ is the
	divergence of the vector field $X$, which is independent of the
	metric since $X(y)=0$. It is just given by the total weight of
	the action on the normal bundle of $H^{-1}(0)$, or alternatively
	the weight of the action on the anticanonical bundle at $y$.
	This is independent of the choice of $y\in H^{-1}(0)$, and we
	denote it by $K$. It follows that 
	\[ \lim_{t\to\infty}(f_t^{-1})^*(\Delta H)(x) = K, 
	\text{ for a.e. }x\in M.\]
	Since $(f_t^{-1})^*(\Delta H)$ is uniformly bounded, independent
	of $t$, it follows that
	\[ \lim_{t\to\infty}
	\int_M (f_t^{-1})^*(\Delta H)\,\omega^n =
	K\int_M\,\omega^n.\]
	At the same time
	\[\frac{d}{dt}\mathcal{M}(\omega_t) =  \int_M
	\dot{\phi_t}(n-S(\omega_t))\,\omega_t^n = \int_M
	H(n-S(\omega))\, \omega^n.\] 
	The proposition follows. 
\end{proof}

\subsection{$\mathbf{P}^2$ blown up in one point}
Let $M_1$ be $\mathbf{P}^2$ blown up in one point. In this section we
prove the following.
\begin{thm} $R(M_1)=6/7$.
\end{thm}

We first show that $R(M_1)\leqslant 6/7$ using twisted slope stability.
Since we give an alternative proof, we will be brief.
Let us write $E$ for the exceptional divisor. We are using the
polarisation
$c_1(M_1)=\mathcal{O}(3)-E$. The Seshadri constant of $E$ is then 2.
If there is a metric $\omega$ with $\mathrm{Ric}(\omega) > t\omega$,
then by taking the trace, we must have
\[ S(\omega) - (1-t)\Lambda_\omega\alpha = 2t\]
for some positive form $\alpha\in c_1(M_1)$. According to
Stoppa~\cite{Sto08_1} (Section 5.1) we have
\[\begin{gathered}
	\frac{3}{2}\frac{2c_1(M_1).E - 2\big[(-c_1(M_1) + (1-t)c_1(M_1)).E +
E^2\big]}{2(3 c_1(M_1).E - 2E^2)} \\ \geqslant \frac{ -(-c_1(M_1)+
(1-t)c_1(M_1)).c_1(M_1)}{c_1(M_1)^2}.\end{gathered}\]
Computing this and simplifying, we obtain exactly $t\leqslant 6/7$. 

Alternatively for a more self-contained proof we can use
Proposition~\ref{prop:limJ}. It is easiest to compute in terms of toric
geometry. The moment polytope of $M_1$ has vertices
$(0,0),(2,0),(2,1),(0,3)$. We choose $H(x,y)=-x$. Then for any toric
metric $\omega\in c_1(M_1)$, using Donaldson's
formula~\cite{Don02} we have
\[ \int_{M_1} H(2-S(\omega))\,\frac{\omega^2}{2!} = 2\int_P H\,d\mu -
\int_{\partial P} H\,d\sigma = -\frac{2}{3},\]
where $d\mu$ is the Lebesgue measure on $P$ and $d\sigma$ is a multiple
of the Lebesgue measure on each edge of $P$, as described in
\cite{Don02}. The weight $K=1$, so 
\[ s\int_{M_1} H(2-S(\omega))\,\frac{\omega^2}{2!} 
+ (1-s)K\,\mathrm{Vol}(M_1) = 4-\frac{14}{3}s.\]
This is negative for $s > 6/7$, in which case
the functional $\mathcal{M} + (1-s)\mathcal{J}_\omega$ is not bounded
below by Proposition~\ref{prop:limJ}, and so 
$R(M_1)\leqslant 6/7$.  

To show that $R(M_1)\geqslant 6/7$ we explicitly construct metrics with
$\mathrm{Ric}(\omega) > t\omega$ for all $t < 6/7$. Again 
thinking of $M_1$ as the $\mathbf{P}^1$ bundle
$\mathbf{P}(\mathcal{O}(-1)\oplus\mathcal{O})$, we will use the momentum
construction to obtain metrics on $M_1$ (for more details on this
construction see \cite{HS02}). Let $\omega_0$ be the
Fubini-Study metric on $\mathbf{P}^1$, and let $h$ be a Hermitian metric
on $\mathcal{O}(-1)$ with curvature form $i\omega_0$. Write
$p:\mathcal{O}(-1)\to\mathbf{P}^1$ for the projection map. On the
complement of the zero section in the
total space of $\mathcal{O}(-1)$ define the metric
\[ \omega = p^*\omega_0 + 2i\partial\bar{\partial}f(s),\]
where $s=\frac{1}{2}\log |z|_h^2$ and $f(s)$ is a suitably convex
function. We change coordinates to $\tau = f'(s)$, which is the moment
map for the $S^1$-action rotating the fibres of $\mathcal{O}(-1)$. Let
$I\subset\mathbf{R}$ be the image of $\tau$ and let $F:I\to\mathbf{R}$
be the Legendre transform of $f$. In other words $F$ is defined by the
equation 
\[f(s) + F(\tau) = s\tau.\]
We then define the \emph{momentum profile} of the metric $\omega$ to be 
\[ \phi(\tau) = \frac{1}{F''(\tau)}.\]
We can compute the Ricci curvature of $\omega$ in terms of $\phi(\tau)$.
In addition if $\phi$ has suitable behaviour at the
endpoints of $I$, then the metric $\omega$ can be extended across the
zero and infinity sections, and we obtain a metric on $M_1$. This is
summarised in the following proposition. For more details see
\cite{HS02} (or also \cite{GSz07_1}). 
\begin{prop}
	Let $\phi:[0,2]\to\mathbf{R}$ be a smooth function such that
	$\phi$ is positive on $(0,2)$, and 
	\[ \phi(0)=\phi(2)=0,\quad \phi'(0)=2,\quad \phi'(2)=-2.\]
	Then we obtain a metric $\omega_\phi\in c_1(M_1)$, given in
	suitable local coordinates by 
	\[ \omega_\phi = (1+\tau)p^*\omega_0 + \phi(\tau)\frac{i\,
	dw\wedge d\overline{w}}{2|w|^2},\]
	and whose Ricci form is
	\[ \rho_\phi = \left(2-\frac{\left[(1+\tau)\phi\right]'}{
	2(1+\tau)}\right)p^*\omega_0
	- \phi \cdot
	\left\{\frac{\left[(1+\tau)\phi\right]'}{2(1+\tau)}\right\}' \cdot
	\frac{i\,
	dw\wedge d\overline{w}}{2|w|^2},\]
	where the primes mean differentiating with respect to $\tau$.
\end{prop}
In order to have $\rho_\phi \geqslant t\omega_\phi$ we need to satisfy two
inequalities
\[ \begin{gathered}
	2 - \frac{\left[(1+\tau)\phi\right]'}{2(1+\tau)} \geqslant t(1+\tau) \\
	-\left\{\frac{\left[(1+\tau)\phi\right]'}{2(1+\tau)}\right\}'
	\geqslant t. 
\end{gathered}\]
By integrating once, it is easy to see that the second inequality
implies the first one for
all $t\leqslant 1$. 

Let $t=6/7$ and let us solve the case of equality in the second
inequality. We obtain a $\psi$ such that
\[ (1+\tau)\psi(\tau) = 2\tau + \frac{1}{7}\tau^2 - \frac{4}{7}\tau^3,\] 
and which satisfies the boundary conditions
\[ \psi(0)=\psi(2)=0,\quad \psi'(0)=2,\quad \psi'(2)=-\frac{10}{7}.\]
Now let 
$\phi(\tau) = \psi(\tau) + \eta(\tau)$, where $\eta$ satisfies
\[ \eta(0)=\eta(2)=0,\quad \eta'(0)=0,\quad \eta'(2)=-\frac{4}{7}.\]
For any $\delta > 0$ we can choose $\eta$ so that for all $\tau$ we have
\[\eta(\tau)\geqslant0,\quad \eta'(\tau),\eta''(\tau)<\delta.\]
Then $\phi=\psi + \eta$ satisfies the boundary conditions that we want, and
\[\begin{aligned}
	-\left\{\frac{\left[(1+\tau)\phi\right]'}{2(1+\tau)}\right\}' 
	&= \frac{6}{7} -\left\{\frac{\left[(1+\tau)
	\eta\right]'}{2(1+\tau)}\right\}'  \\
	&= \frac{6}{7}-\frac{1}{2}\eta''(\tau) - 
	\frac{\eta'(\tau)}{2(1+\tau)} +
	\frac{\eta(\tau)}{2(1+\tau)^2} \\
	&> \frac{6}{7} - \delta.
\end{aligned}\]
Letting $\delta\to 0$, we find that we can obtain a metric with
$\mathrm{Ric}(\omega)\geqslant t\omega$ for all $t < 6/7$, so $R(M_1)\geqslant
6/7$. Note that we would have to analyse the metrics more carefully near
$\tau=0$ and $\tau=2$ to see whether we have the strict inequality, but
clearly $\mathrm{Ric}(\omega)\geqslant t\omega$ is enough for what we
want. This completes the proof that $R(M_1)=6/7$. 

Note that in the limiting case $t=6/7$, the function $\psi$ that we 
found above defines a singular
metric satisfying $\mathrm{Ric}(\omega)\geqslant \frac{6}{7}\omega$. The
fact that $\psi'(2)=-10/7$ means that the metric has conical
singularities with angle $2\sin^{-1}\sqrt{5/7}$ along a line not meeting
the exceptional divisor (ie. along the divisor at
infinity in $\mathbf{P}(\mathcal{O}(-1)\oplus\mathcal{O})$).

\subsection{$\mathbf{P}^2$ blown up in two points}
Let $M_2$ be $\mathbf{P}^2$ blown up in two points. In this section we
prove
\begin{prop} $1/2\leqslant R(M_2)\leqslant 21/25$.
\end{prop}
In this case twisted 
slope stability will not give any obstruction, since $M_2$ is slope
stable (see
Panov-Ross~\cite{PR07}) and our twisting just makes things more stable (we
are adding a proper function to $\mathcal{M}$). However we can apply
Proposition~\ref{prop:limJ}. Once again we work in terms of the toric
polygon to make the computations easier. The polygon corresponding to
$M_2$ has vertices $(0,0),(2,0),(2,1),(1,2),(0,2)$. We let
$H(x,y)=-x-y$, so the weight $K=1$. 
As before, using Donaldson's formulae we obtain
\[ s\int_{M_2} H(2-S(\omega))\,\frac{\omega^2}{2} + (1-s)K\,\mathrm{Vol}
(M_2) =
\frac{7}{2} - \frac{25}{6}s.\]
This is negative if $s > 21/25$, so by Proposition~\ref{prop:limJ}
we obtain $R(M_2)\leqslant
21/25$.

To show that $R(M_2)\geqslant 1/2$ we use the $\alpha$-invariant.
According to Song~\cite{Song05} the $\alpha$-invariant for torus
invariant K\"ahler potentials on $M_2$ is $1/3$. It follows (see
\cite{Tian87}) that
$R(M_2)\geqslant 1/3\cdot 3/2 = 1/2$. It would be very interesting to
find better bounds on $R(M_2)$. 

\section{More general test-configurations}\label{sec:test-config}
We have seen in Proposition~\ref{prop:def} that if $M$ is
K\"ahler-Einstein then $R(M)=1$ but the converse is not true. In this
section we show the following weaker converse.
\begin{thm}\label{thm:semistab}
	If $R(M)=1$, then $M$ is K-semistable with respect to 
	test-con\-fi\-gu\-ra\-tions with smooth total space.
\end{thm}
Before giving the proof we briefly explain K-semistability. A
test-configuration $\chi$ for $M$ is a flat polarised family
$\pi:(\mathcal{M},\mathcal{L})\to\mathbf{C}$, such that
\begin{itemize}
	\item $\pi$ is $\mathbf{C}^*$-equivariant, 
	\item $\mathcal{L}$ is relatively ample,
	\item we have
\[ (\mathcal{M}_t,\mathcal{L}|_{\mathcal{M}_t})\cong (M,(-K_M)^k) \]
for $t\not=0$ and some integer $k>0$. 
\end{itemize}
The central fibre is then a polarised scheme $(M_0,L_0)$, with a
$\mathbf{C}^*$ action. This allows us to define the Futaki invariant
$F(\chi)$ of the test-configuration, which generalises the classical
Futaki invariant in case $M_0$ is smooth and the $\mathbf{C}^*$-action
is generated by a holomorphic vector field. For details see
Donaldson~\cite{Don02}. The manifold $M$ is called \emph{K-semistable}
if $F(\chi)\geqslant 0$ for all test-configurations $\chi$. If in
addition $F(\chi)=0$ only for test-configurations where the central
fibre is isomorphic to $M$, we say that $M$ is \emph{K-polystable}. The
central conjecture is
\begin{conj}[Yau-Tian-Donaldson Conjecture] The manifold $M$ admits a
	K\"ahler-Einstein metric if and only if $M$ is K-polystable.
\end{conj}
In light of this it is reasonable to conjecture the following.
\begin{conj} The Fano manifold $M$ is K-semistable if and only if
	$R(M)=1$. 
\end{conj}
Our Theorem \ref{thm:semistab} goes some way in proving the easier
direction of this conjecture. 
\begin{proof}[Proof of Theorem \ref{thm:semistab}]
	Suppose we have a test-configuration for $M$ with total space
	$\mathcal{M}$. We can realise it
	as a one parameter group acting on an embedding in projective
	space. More precisely we have
	an embedding $F:M\to\mathbf{P}^N$ and a $\mathbf{C}^*$-action on
	$\mathbf{P}^N$. Choose a Fubini-Study metric $\omega_{FS}$ on
	$\mathbf{P}^N$ which in invariant under $S^1$, and let $H$ be
	a Hamiltonian function of this $S^1$-action, normalised so that
	$\sup H=0$. Let us write $f_t:\mathbf{P}^N\to\mathbf{P}^N$ for
	the gradient flow of $\nabla H$. We then have a family of
	metrics
	\[ \omega_t = F^*(f_t^*\omega_{FS}) \]
	on $M$ and we let $\omega=\omega_0$. 
	Suppose that the Futaki invariant of the
	test-configuration is negative, ie. $M$ is not K-semistable. We
	want to show that $R(M)<1$. Since the total space of the
	test-configuration is smooth, according to \cite{PRS06} (see also
	\cite{PT06}) we have
	\[ \limsup_{t\to\infty}\frac{d}{dt}\mathcal{M}(\omega_t) < 0.\]
	We want to show that for suitably small $\epsilon>0$ we have
	\begin{equation}\label{ineq:epsilon}
		\limsup_{t\to\infty}\frac{d}{dt}(\mathcal{M}(\omega_t) +
		\epsilon\mathcal{J}_\omega(\omega_t)) < 0,
	\end{equation}
	which will imply that $R(M)\leqslant 1-\epsilon$. 

	To show the inequality (\ref{ineq:epsilon}) we compute
	\[ \frac{d}{dt}\mathcal{J}_\omega(\omega_t) = \int_M
	\dot{\phi_t}(\Lambda_{\omega_t}\omega-n)\,\omega_t^n.\]
	Note that $\dot{\phi_t} =F^*(f_t^*H)$ and since $H\leqslant 0$,
	we have
	\[ \frac{d}{dt}\mathcal{J}_\omega(\omega_t) \leqslant -n\int_M
	F^*(f_t^*H)\,\omega_t^n \leqslant -n\,\mathrm{Vol}(M)\,\inf H.\]
	Thus the limit as $t\to\infty$ is bounded above, so for suitably small
	$\epsilon>0$ we have (\ref{ineq:epsilon}).
\end{proof}

\bibliographystyle{hacm}
\bibliography{../mybib}

\begin{thebibliography}{10}

\bibitem{Aub78}
{\sc Aubin, T.}
\newblock {\'E}quations du type {M}onge-{A}mp\`ere sur les variet\'es
  k\"ahl\'eriennes compactes.
\newblock {\em Bull. Sci. Math. (2) 102}, 1 (1978), 63--95.

\bibitem{Aub84}
{\sc Aubin, T.}
\newblock R\'eduction de cas positif de l'\'equation de {M}onge-{A}mp\`ere sur
  les vari\'et\'es k\"ahl\'eriennes compactes \`a la d\'emonstration d'une
  in\'egalit\'e.
\newblock {\em J. Funct. Anal. 57}, 2 (1984), 143--153.

\bibitem{BM85}
{\sc Bando, S., and Mabuchi, T.}
\newblock Uniqueness of {E}instein {K}\"ahler metrics modulo connected group
  actions.
\newblock In {\em Algebraic geometry, Sendai\/} (1985), vol.~10 of {\em Adv.
  Stud. Pure Math.}, pp.~11--40.

\bibitem{Chen00}
{\sc Chen, X.~X.}
\newblock On the lower bound of the {M}abuchi energy and its application.
\newblock {\em Int. Math. Res. Notices 12\/} (2000), 607--623.

\bibitem{Chen08}
{\sc Chen, X.~X.}
\newblock Space of {K}\"ahler metrics {IV} -- on the lower bound of the
  {K}-energy.
\newblock {\em preprint\/} (2008), arXiv:0809.4081.

\bibitem{CT05_1}
{\sc Chen, X.~X., and Tian, G.}
\newblock Geometry of {K}\"ahler metrics and foliations by holomorphic discs.
\newblock {\em preprint\/}, math.DG/0507148.

\bibitem{Don08_1}
{\sc Donaldson, S.~K.}
\newblock Constant scalar curvature metrics on toric surfaces.
\newblock arXiv:0805.0128.

\bibitem{Don01}
{\sc Donaldson, S.~K.}
\newblock Scalar curvature and projective embeddings, {I}.
\newblock {\em J. Differential Geom. 59\/} (2001), 479--522.

\bibitem{Don02}
{\sc Donaldson, S.~K.}
\newblock Scalar curvature and stability of toric varieties.
\newblock {\em J. Differential Geom. 62\/} (2002), 289--349.

\bibitem{Don07}
{\sc Donaldson, S.~K.}
\newblock A note on the $\alpha$-invariant of the {M}ukai-{U}memura 3-fold.
\newblock arXiv:0711.4357.

\bibitem{HS02}
{\sc Hwang, A., and Singer, M.~A.}
\newblock A momentum construction for circle-invariant {K}\"ahler metrics.
\newblock {\em Trans. Amer. Math. Soc. 354}, 6 (2002), 2285--2325.

\bibitem{Kur65}
{\sc Kuranishi, M.}
\newblock New proof for the existence of locally complete families of complex
  structures.
\newblock In {\em Proc. Conf. Complex Analysis (Minneapolis, 1964)\/} (1965),
  Springer, Berlin, pp.~142--154.

\bibitem{Mab86}
{\sc Mabuchi, T.}
\newblock K-energy maps integrating {F}utaki invariants.
\newblock {\em Tohoku Math. J. 38}, 4 (1986), 575--593.

\bibitem{MS98}
{\sc McDuff, D., and Salamon, D.}
\newblock {\em Introduction to symplectic topology}.
\newblock OUP, 1998.

\bibitem{PR07}
{\sc Panov, D., and Ross, J.}
\newblock Slope stability and exceptional divisors of high genus.
\newblock {\em preprint\/}, arXiv:0710.4078.

\bibitem{PT06}
{\sc Paul, S.~T., and Tian, G.}
\newblock {CM} stability and the generalised {F}utaki invariant {II},
  math.DG/0606505.

\bibitem{PRS06}
{\sc Phong, D.~H., Ross, J., and Sturm, J.}
\newblock Deligne pairings and the {K}nudsen-{M}umford expansion.
\newblock {\em J. Differential Geom. 78}, 3 (2008), 475--496.

\bibitem{PSSW}
{\sc Phong, D.~H., Song, J., Sturm, J., and Weinkove, B.}
\newblock The {M}oser-{T}rudinger inequality on {K}\"ahler-{E}instein
  manifolds.
\newblock {\em Amer. J. Math. 130}, 4 (2008), 1067--1085.

\bibitem{PS08}
{\sc Phong, D.~H., and Sturm, J.}
\newblock Lectures on stability and constant scalar curvature.
\newblock arXiv:0801.4179.

\bibitem{RT06}
{\sc Ross, J., and Thomas, R.~P.}
\newblock An obstruction to the existence of constant scalar curvature
  {K}\"ahler metrics.
\newblock {\em J. Differential Geom. 72\/} (2006), 429--466.

\bibitem{Song05}
{\sc Song, J.}
\newblock The $\alpha$-invariant on toric {F}ano manifolds.
\newblock {\em Amer. J. Math. 127}, 6 (2005), 1247--1259.

\bibitem{ST07}
{\sc Song, J., and Tian, G.}
\newblock The {K}\"ahler-{R}icci flow on surfaces of positive {K}odaira
  dimension.
\newblock {\em Invent. Math. 170}, 3 (2007), 609--653.

\bibitem{SW04}
{\sc Song, J., and Weinkove, B.}
\newblock On the convergence and singularities of the {J}-flow with
  applications to the {M}abuchi energy.
\newblock {\em Comm. Pure Appl. Math. 61}, 2 (2008), 210--229.

\bibitem{Sto08_1}
{\sc Stoppa, J.}
\newblock Twisted csc{K} metrics and {K}\"ahler slope stability.
\newblock {\em preprint\/}, arXiv:0804.0414.

\bibitem{GSz07_1}
{\sc Sz\'ekelyhidi, G.}
\newblock The {C}alabi functional on a ruled surface.
\newblock {\em to appear in Ann. Sci. \'Ec. Norm. Sup\'er.\/}, math.DG/0703562.

\bibitem{GSz08}
{\sc Sz\'ekelyhidi, G.}
\newblock The {K}\"ahler-{R}icci flow and {K}-polystability.
\newblock {\em to appear in Amer. J. Math.\/}, arXiv:0803.1613.

\bibitem{Tian87}
{\sc Tian, G.}
\newblock On {K}\"ahler-{E}instein metrics on certain {K}\"ahler manifolds with
  $c_1({M})>0$.
\newblock {\em Invent. Math. 89\/} (1987), 225--246.

\bibitem{Tian92}
{\sc Tian, G.}
\newblock On stability of the tangent bundles of {F}ano varieties.
\newblock {\em Internat. J. Math. 3}, 3 (1992), 401--413.

\bibitem{Tian97}
{\sc Tian, G.}
\newblock K\"ahler-{E}instein metrics with positive scalar curvature.
\newblock {\em Invent. Math. 137\/} (1997), 1--37.

\bibitem{Wei04}
{\sc Weinkove, B.}
\newblock Convergence of the {J}-flow on {K}\"ahler surfaces.
\newblock {\em Comm. Anal. Geom. 12}, 4 (2004), 949--965.

\bibitem{Yau78}
{\sc Yau, S.-T.}
\newblock On the {R}icci curvature of a compact {K}\"ahler manifold and the
  complex {M}onge-{A}mp\`ere equation {I}.
\newblock {\em Comm. Pure Appl. Math. 31\/} (1978), 339--411.

\bibitem{Yau93}
{\sc Yau, S.-T.}
\newblock Open problems in geometry.
\newblock {\em Proc. Symposia Pure Math. 54\/} (1993), 1--28.

\end{thebibliography}
\end{document}